\documentclass[a4paper,12pt]{article}

\usepackage[margin=1in]{geometry}
\usepackage[T1]{fontenc}

\usepackage{a4wide}
\usepackage{epstopdf}
\usepackage{epsfig}
\usepackage{curves}
\usepackage{epic,eepic}
\usepackage{graphicx}
\usepackage{amsmath,amsfonts,amssymb,latexsym}
\usepackage{mathabx}
\usepackage{enumerate}
\usepackage{multirow}
\usepackage{rotating}
\usepackage[thmmarks,amsmath]{ntheorem}
\usepackage{lscape}
\usepackage{float}
\usepackage[round]{natbib}
\floatstyle{boxed}
\usepackage{authblk}
\usepackage[ruled,lined,boxed]{algorithm2e}
\usepackage{epstopdf}
\usepackage{graphics}

\usepackage{titlesec}
\titleformat{name=\section}[block]{\bfseries\filcenter}{\thesection}{1em}{}
\titleformat{name=\subsection}[block]{\bfseries\filcenter}{\thesubsection}{1em}{}

\SetKwBlock{Repeat}{repeat}{end}
\allowdisplaybreaks

\newtheorem{proposition}{Proposition}[section]

\theorembodyfont{\normalfont}

\newtheorem{remark}[proposition]{Remark}

\newenvironment{proof}{\noindent\textbf{Proof.}
  }{\hspace*{\fill}$\Box$ \\[1em]}

\renewcommand{\Re}{{\mathbb R}}
\def\trt{^{\scriptscriptstyle T}}

\title{A dominance maximization approach to portfolio selection}

\author[1]{Francesco Cesarone}
\author[2]{Lorenzo Lampariello}       
\author[3]{Simone Sagratella} 

\affil[1]{Department of Business Studies, Roma Tre University, Rome, Italy\\ francesco.cesarone@uniroma3.it}
           
\affil[2]{Department of Business Studies, Roma Tre University, Rome, Italy\\ lorenzo.lampariello@uniroma3.it}
          
\affil[3]{Department of Computer, Control and Management Engineering Antonio Ruberti, Sapienza University of Rome, Rome, Italy\\

sagratella@diag.uniroma1.it}

\begin{document}

\maketitle



\maketitle

\begin{abstract}
In the portfolio multiobjective optimization framework, we propose to compare and choose, among all feasible asset portfolios of a given market, the one that maximizes the product of the distances between its values of risk and gain and those of a suitable reference point (e.g., the so-called nadir). We show that this approach has distinctive and remarkable features. While being not influenced by how the objectives are scaled, it provides one with an efficient (Pareto) portfolio that ``dominates the most'' with respect to the reference point. Furthermore, although our no-preference strategy generally requires the solution of a nonconvex (constrained) single-objective problem, we show how the resulting (global) optimal portfolio can be easily and efficiently computed. We also perform numerical tests based on some publicly available benchmark data sets often used in the literature, highlighting the nice properties of our approach.
\end{abstract}


\section{Introduction}
Portfolio selection is a typical decision problem under uncertainty,
where the drivers of uncertainty are the asset returns.
Generally, the aim is to choose the fractions of a given capital invested in each of $n$ assets,
such that the resulting portfolio return satisfies specific criteria.
Starting with the seminal work of \citet[]{markowitz1952portfolio,Mark:59},
the key problem in asset allocation is to select a portfolio with appropriate features
in terms of gain and risk.
From a mathematical viewpoint, the synthetic indices that represent gain and risk are modeled by functions of $n$ real variables to be optimized simultaneously: this results in a two objectives problem. Solving multiobjective programs usually consists in computing the Pareto optimal solution that best suites the decision maker; this is in general carried out by scalarization. Just with respect to the latter approach, it is hardly possible here to even summarize the huge amount of solution methods that have been proposed in the literature. We only mention that solution methods are commonly grouped into four main categories: {\it{no-preference}}, {\it{a posteriori}}, {\it{a priori}} and {\it{interactive}} methods. We refer the reader to \cite{miettinen2012nonlinear} for both theoretical bases and review of the related literature.

In portfolio selection, typically, one can distinguish a first step of the gain-risk analysis \citep[see, e.g.,][]{Mark:59,Elton2009modern,blay2013risk} where the efficient portfolios, namely the Pareto optimal solutions, are detected. Among the Pareto optimal portfolios, one can adopt, with respect to risk or gain, preference criteria that, possibly, are introduced in the second step of the gain-risk analysis. For example, 
once the preferences of the decision maker have been modeled by a utility function, the optimal portfolio must be the efficient one with maximum expected utility \citep[see][and references therein]{blay2013risk,Markowitz2014mean,Carleo2017approximating}. Here, for the sake of completeness, we also mention other related works \citep[see][and references therein]{Roman2009,Bruni2012new,Bruni2015linear}. 

One of the main limitations of the expected utility maximization is the subjective specification of a utility function. In this paper, adopting a different point of view, we wish to provide an alternative framework to the expected utility approach, in the same spirit of compromise programming techniques \citep[cf., e.g.,][]{Ballestero1996portfolio}. More precisely, along the lines of the analysis in \cite{stoyanov2007optimal}, we propose to compare and choose among all portfolios by relying on a performance measure optimization: this results in a no-preference strategy that requires the solution of a nonlinear nonconvex (constrained) single-objective problem. 

To easily describe the rationale behind this no-preference method, let us consider a (convex) measure of risk $\rho_{P}(x)$
and one (concave) that represents gain $\gamma_{P}(x)$, where $x \in \mathbb R^n$ is the vector of portfolio weights.
Each feasible portfolio $x$ identifies, in the risk-gain plane, a rectangle $R_x$ defined by a suitable reference point $(\rho_P^{\text{ref}}, \gamma_P^{\text{ref}})$
(consider, here, for the sake of simplicity, the vector whose components are the maximal risk and the minimal gain over the feasible region, respectively) and the point $\left(\rho_{P}(x), \, \gamma_{P}(x)\right).$
Clearly, $x$ dominates all the feasible portfolios whose objective values belong to the rectangle.
In the light of this consideration, we aim at computing a Pareto optimal solution that maximizes the area of the corresponding rectangle.
This simple idea draws inspiration from the {\it{hypervolume}} paradigm \citep[see, e.g.,][]{auger2009theory,fleischer2003measure,zitzler1998multiobjective}, shares some conceptual similarities with other approaches \citep[e.g., GUESS and the method of the global criterion][]{buchanan1997naive,miettinen2012nonlinear}, and, from another point of view, consists in maximizing the (weighted) geometric mean of the difference between the objective functions and the components of the reference point \citep[see][]{audet2008multiobjective,lootsma1995controlling}.

We further investigate the peculiarity of the resulting approach: we show that using the area of the rectangle $R_x$ as performance measure in order to compare portfolios has distinctive and significant features. Namely, computing a (global) solution of the underlying optimization problem is not sensitive to the scaling of the objective functions and provides one with an efficient portfolio that ``dominates the most'' (in the sense of the area of the induced rectangle) with respect to the risk-gain space. Interestingly, this Pareto optimum has the nice property that any other feasible portfolio (including all the other efficient ones), for which one of the two objectives is ``improved'' (e.g., entailing a larger gain or a smaller risk than the ones provided by our approach) by a factor, is such that the other objective must ``worsen'' {\textit{by at least the same factor}}. 
Moreover, one of the main practical advantages of our approach is that the underlying optimization problem turns out to be easily solvable. In fact, while, in general, one can not expect to solve globally a nonconvex single-objective program (like the one arising in our framework) in order to recover global Pareto optimal solutions, however, regarding the problem of maximizing the area of $R_x$, any stationary point, with a nontrivial positive value of the corresponding area, turns out to be global (Pareto) optimal. Leveraging this remarkable property, we present a very effective variant of the projected gradient algorithm to calculate efficiently the sought global (Pareto) optimum.

We provide some numerical results showing the significance of the computed Pareto optima. We stress that, for each stock market that we have considered, the amount of time needed in order to reach a solution is below half a second.

With the preliminary results in this work, we intend to lay down the basis for further analyses and developments. On the one hand, we wish to compare, from a theoretical standpoint, our method with other classical techniques aimed at optimizing performance measures such as the reward-risk ratios (cf., e.g., Sharpe, STARR and Rachev ratios) or the Chebychev norm distance. On the other hand, referring to Nash equilibrium problems contexts \citep[see, e.g.,][]{aussel2017sufficient,dreves2011solution,facchinei2011partial,FacchSagra10,sagratella2016computing,sagratella2017algorithms,sagratella2017computing,scutari2012equilibrium}, we envisage that the peculiar nature of the approach makes it fit particularly well into the noncooperative scenario of multi-portfolio selection.

\section{Main properties of the approach}\label{sec:prel}
We consider the portfolio selection problem
\begin{equation}\label{eq:mobini}
\begin{array}{cl}
\underset{x}{\mbox{minimize}} & \big(-\gamma_P(x), \, \rho_P(x)\big)\trt\\
\mbox{s.t.} & x \in \Delta,
\end{array}
\end{equation}
where $\gamma_P(x): \Re^n \to \Re$ is any continuous concave measure of gain and $\rho_P(x): \Re^n \to \Re$ is any continuous convex measure of risk; moreover, $\Delta \triangleq \left\{x \in \mathbb R^n \, : \, x \ge 0, \, e\trt x = 1\right\}$, where $e \in \mathbb R^n$ is the vector with all components being $1$, is the nonempty, convex and compact feasible set. For the sake of simplicity and without loss of generality, we do not contemplate short sales. 

Let $(\rho_P^{\text{ref}}, \gamma_P^\text{ref}) \in \mathbb R^2$ be a suitable reference point, such as the so-called nadir vector, whose components are the individual maxima in the Pareto front of each objective (that is, as for $\rho_P^\text{ref}$, the risk of the portfolio with maximal gain, while, as $\gamma_P^\text{ref}$, the gain of the minimum variance portfolio), or the worst point, that is the vector whose components are the maximal risk and the minimal gain over the feasible set, respectively: any portfolio $x \in \Delta$, such that $\gamma_P(x) \ge \gamma_P^{\text{ref}}$ and $\rho_P(x) \le \rho_P^{\text{ref}}$, identifies in the risk-gain space a rectangle $R_x$ with base $\left(\rho_P^\text{ref} - \rho_P(x)\right) \ge 0$ and height $\left(\gamma_P(x) - \gamma_P^\text{ref}\right) \ge 0$. 

We compare portfolios by maximizing, as performance measure, the area of $R_x$: 
\begin{equation}\label{eq:mobour}
\begin{array}{cl}
\underset{x}{\mbox{minimize}} & -\big(\gamma_P(x) - \gamma_P^\text{ref}\big) \big(\rho_P^\text{ref} - \rho_P(x)\big)\\
\mbox{s.t.} & \gamma_P(x) \ge \gamma_P^\text{ref}\\[5pt]
& \rho_P(x) \le \rho_P^\text{ref}\\[5pt]
&  x \in \Delta.
\end{array}
\end{equation}
We denote by $A(x) \triangleq \big(\gamma_P(x) - \gamma_P^\text{ref}\big) \big(\rho_P^\text{ref} - \rho_P(x)\big)$ and $X \triangleq \{x \in \Delta \, : \,  \gamma_P(x) \ge \gamma_P^\text{ref}, \, \rho_P(x) \le \rho_P^\text{ref}\}$ the area of $R_x$, which is maximized in \eqref{eq:mobour}, and the feasible set in \eqref{eq:mobour}, respectively. Note that, from another point of view, $A(x)$ is nothing else than the geometric mean, with unitary weights, of the differences between the components of a suitable reference point and the objectives, \citep[see][]{lootsma1995controlling}.
It goes without saying that, if $\rho_P^{\text{ref}} = \max \{\rho_P(x) \, : \, x \in \Delta\}$ and $\gamma_P^{\text{ref}} = \min \{\gamma_P(x) \, : \, x \in \Delta\}$, i.e. the reference point is the worst one, the constraints $\gamma_P(x) \ge \gamma_P^{\text{ref}}$ and $\rho_P(x) \le \rho_P^{\text{ref}}$ in \eqref{eq:mobour} are satisfied for every $x \in \Delta$, and, thus, can be dropped. We indicate with $x_A$ a global optimal point for problem \eqref{eq:mobour}. In a nutshell, choosing among portfolios by addressing \eqref{eq:mobour} has the following prominent advantages and distinctive features:
\begin{enumerate}
\item[(i)]
$x_A$ is a global Pareto efficient portfolio for \eqref{eq:mobini} (see Proposition \ref{th:ottimo-pareto} and \ref{pr: area > 0});
\item[(ii)]
among all efficient portfolios for \eqref{eq:mobini}, ${x_A}$, maximizing the area of the rectangle $R_x$, ``dominates the most'' in the risk-gain space with respect to the reference point $(\rho_P^\text{ref}, \gamma_P^\text{ref})$ (see Figure \ref{fig:area} for a rather explanatory illustration of this feature);   
\item[(iii)]
finding a solution of \eqref{eq:mobour} does not depend on the objectives' scales, thus allowing to ``robustly'' deal with non homogeneous objectives, such as risk and gain (see Remark \ref{rm:ii});
\item[(iv)]
any feasible portfolio (including all the efficient ones), for which one between risk and gain is ``improved'' (w.r.t. its value at ${x_A}$) by a factor, is such that the remaining objective must ``worsen'' (w.r.t. its value at ${x_A}$) {\textit{by at least the same factor}} (see Proposition \ref{th:proper});
\item[(v)]
$x_A$ is easily and efficiently computable (see Section \ref{sec:theory}).
 \end{enumerate} 
With the following proposition, which is reminiscent of a similar result in \cite{audet2008multiobjective}, we prove the claim in (i) establishing the link between \eqref{eq:mobour} and the original two objective portfolio selection \eqref{eq:mobini}.
\begin{proposition}\label{th:ottimo-pareto}
Any global solution ${x_A}$ for problem \eqref{eq:mobour} such that $A({x_A}) > 0$ is a global Pareto optimum for problem \eqref{eq:mobini}.
\end{proposition}
\begin{proof}
Suppose by contradiction that there exists $\widehat x \in X$ such that, without loss of generality, $\gamma_P(\widehat x) > \gamma_P({x_A}), \; \rho_P(\widehat x) \le \rho_P({x_A}).$
In turn, $\gamma_P(\widehat x) - \gamma_P^\text{ref} > \gamma_P({x_A}) - \gamma_P^\text{ref}>0, \; \rho_P^\text{ref} - \rho_P(\widehat x) \ge \rho_P^\text{ref} - \rho_P({x_A}) > 0$ and, hence, $A(\widehat x) > A({x_A})$, which contradicts the optimality of ${x_A}$ for problem \eqref{eq:mobour}.  \end{proof}
We remark that the assumption $A({x_A}) > 0$ in Proposition \ref{th:ottimo-pareto} is not demanding since the optimal value of problem \eqref{eq:mobour} reaches zero only in pathological cases. Specifically, in Proposition \ref{pr: area > 0} we show that, if condition $A({x_A}) > 0$ is not verified, then at least an objective reaches its reference value for every feasible point, and, thus, can be dropped.
\begin{proposition}\label{pr: area > 0}
Let ${x_A}$ be any global solution for problem \eqref{eq:mobour}. Then, $A({x_A}) = 0$ if and only if either $\gamma_P(x) = \gamma_P^\text{ref}$ for all $x \in X$ or $\rho_P(x) = \rho_P^\text{ref}$ for all $x \in X$.
\end{proposition}
\begin{proof}
We only show the necessity. 
 Let $A({x_A}) = \big(\gamma_P({x_A}) - \gamma_P^\text{ref}\big) \big(\rho_P^\text{ref} - \rho_P({x_A})\big) = 0$ and suppose by contradiction that $\widehat x, \widecheck x \in X$ exist such that $\gamma_P(\widehat x) > \gamma_P^\text{ref}$ and $\rho_P(\widecheck x) < \rho_P^\text{ref}$. Clearly, we have $\rho_P(\widehat x) = \rho_P^\text{ref}$ and $\gamma_P(\widecheck x) =  \gamma_P^\text{ref}$ because $A(x) = 0$ for every $x \in X$, due to $A(x_A) = 0$. Setting $\overline x = \frac{1}{2} \widehat x + \frac{1}{2} \widecheck x$, we obtain, by the concavity of $\gamma_P$, $\gamma_P(\overline x) \geq \frac{1}{2} \gamma_P(\widehat x) + \frac{1}{2} \gamma_P^\text{ref} > \gamma_P^\text{ref}$ and, thanks to the convexity of $\rho_P$, $\rho_P(\overline x) \leq \frac{1}{2} \rho_P^\text{ref} + \frac{1}{2} \rho_P(\widecheck x) < \rho_P^\text{ref}$. Therefore, $A(\overline x) = \big(\gamma_P(\overline x) - \gamma_P^{\text{ref}}\big) \big(\rho_P^{\text{ref}} - \rho_P(\overline x)\big)> 0$, in contradiction with the fact that the optimal value $A({x_A})$ is zero.  \end{proof}
In the following remark we better explain property (ii) above.
\begin{remark}\label{rm:ii}
If the decision maker prefers to scale, for example, the risk measure $\rho_P$ by a positive factor $\eta$, then the reference value becomes $\eta \rho_P^{\text{ref}}$ and, $\arg\min \{ -\big(\gamma_P(x) - \gamma_P^\text{ref}\big) \big(\rho_P^\text{ref} - \rho_P(x)\big) \, : \, x \in X\} = \arg\min \{ -\big(\gamma_P(x) - \gamma_P^\text{ref}\big) \eta \big(\rho_P^\text{ref} - \rho_P(x)\big) \, : \, x \in X\}$, for every $\eta > 0$. {\hspace*{\fill}$\Box$}
\end{remark}  
From now on, in the light of the results above, we refer to a generic global solution of problem \eqref{eq:mobour} as the Pareto efficient portfolio $x_A$ of \eqref{eq:mobini}. The following Proposition \ref{th:proper} justifies the claim in (iv).
\begin{proposition}\label{th:proper}
An efficient portfolio $x_A$ enjoys the following properties:
\begin{enumerate}
\item[(a)]
if a feasible portfolio $x \in X$ is such that $\gamma_P(x) - \gamma_P^\text{ref} = \alpha \big(\gamma_P({x_A}) - \gamma_P^\text{ref}\big)$ for some $\alpha \ge 1$, then $\alpha \big(\rho_P^\text{ref} - \rho_P(x)\big) \le \rho_P^\text{ref} - \rho_P({x_A})$; 
\item[(b)]
if a feasible portfolio $x \in X$ is such that $\rho_P^\text{ref} - \rho_P(x) = \beta \big(\rho_P^\text{ref} - \rho_P({x_A})\big)$ for some $\beta \ge 1$, then $\beta \big(\gamma_P(x) - \gamma_P^\text{ref} \big)\le \gamma_P({x_A}) - \gamma_P^\text{ref}$.
\end{enumerate}
\end{proposition}
\begin{proof}
The claims are due to the following relations:
\[
\alpha = \frac{\gamma_P({x}) - \gamma_P^\text{ref}}{\gamma_P({x_A}) - \gamma_P^\text{ref}}\le \frac{\rho_P^\text{ref} - \rho_P({x_A})}{\rho_P^\text{ref} - \rho_P({x})} = \frac{1}{\beta},  
\]
where the inequality holds since $A(x_A) \ge A(x)$ for every $x \in X$. \end{proof}
The result above shows that, relative to the reference point, if one wants to improve an objective (w.r.t. its value at $x_A$) by a factor, then the price to pay is a corresponding deterioration of the other objective (w.r.t. its value at $x_A$) {\textit{by at least the same factor}}: for quite a clear illustration of this feature see the column {\bf improve}$|$worsen in Table \ref{tab:res1} and the related description.

\begin{remark}\label{rm:sr}
As side consideration, we recall that the (relative) substitution rate between risk and gain
\[
\left|\frac{1}{(\rho_P^{\text{ref}} - \rho_P)} \; \frac{\partial A}{\partial \gamma_P}\right| \; \left/ \; \left|\frac{1}{(\gamma_P - \gamma_P^{\text{ref}})} \; \frac{\partial A}{\partial \rho_P}\right| \right.,
\]
along an indifference line of $A$ in the objectives space, is unitary \citep[see, e.g.,][]{lootsma1995controlling}.
 {\hspace*{\fill}$\Box$}
\end{remark}

\section{Solvability issues}\label{sec:theory} 
While the simple ideas set forth in the previous section seem appealing, the question about how one can solve problem \eqref{eq:mobour} still remains open: suffice it to observe that, even with a concave $\gamma_P$ and a convex $\rho_P$, the objective function in \eqref{eq:mobour} is in general nonconvex. In fact, it is only semistrictly quasiconvex, see Theorem 5.15 in \citep{avriel2010generalized}. 

From now on, assume $\gamma_P$ and $\rho_P$ to be continuously differentiable and with Lipschitz continuous gradients. The following proposition, proving that any stationary point, with a corresponding non zero value of the area $A$, is a global optimum for \eqref{eq:mobour} and, in turn, a Pareto efficient portfolio that maximizes $A$, shows the viability and thus, the significance of the approach.
\begin{proposition}\label{th:statglob}
Let $\overline x \in X$ be stationary for problem \eqref{eq:mobour}, that is
\begin{equation}\label{eq:minpri}
-\nabla A(\overline x)\trt (x - \overline x) \ge 0, \; \forall x \in X,
\end{equation}
and such that $A(\overline x) > 0$. Then, $\overline x$ is global optimal for problem \eqref{eq:mobour} and, in turn, a global Pareto optimum for problem \eqref{eq:mobini}.   
\end{proposition}
\begin{proof}
Assume by contradiction that $\widetilde x \in X$ exists such that $A(\widetilde x) > A(\overline x)$. We denote $\overline \gamma_P \triangleq \gamma_P(\overline x) - \gamma_P^\text{ref}$, $\overline \rho_P \triangleq \rho_P^\text{ref} - \rho_P(\overline x)$, $\widetilde \gamma_P \triangleq \gamma_P(\widetilde x) - \gamma_P^\text{ref}$ and $\widetilde \rho_P \triangleq \rho_P^\text{ref} - \rho_P(\widetilde x)$, which are positive quantities because $\overline x, \, \widetilde x \in X$ and
\begin{equation}\label{eq:interstep}
\widetilde \gamma_P \widetilde \rho_P = A(\widetilde x) > A(\overline x) = \overline \gamma_P \overline \rho_P > 0.
\end{equation}
We get the following chain of inequalities that leads to an absurdum:
\[
\begin{array}{rcl}
0 & \overset{(a)}{\le} & - \overline \rho_P \nabla \gamma_P(\overline x)\trt (\widetilde x - \overline x) + \overline \gamma_P \nabla \rho_P(\overline x)\trt (\widetilde x - \overline x)\\
& \overset{(b)}{\le} & \overline \rho_P \big(\gamma_P(\overline x) - \gamma_P(\widetilde x)\big) + \overline \gamma_P \big(\rho_P(\widetilde x) - \rho_P(\overline x)\big)\\
& = & \overline \rho_P (\overline \gamma_P - \widetilde \gamma_P) + \overline \gamma_P (\overline \rho_P - \widetilde \rho_P) = \overline \gamma_P \left(2 \overline \rho_P - \frac{\overline \rho_P \widetilde \gamma_P}{\overline \gamma_P} - \frac{\widetilde \rho_P^2}{\widetilde \rho_P}\right)\\
& \overset{(c)}{<} & - \frac{\overline \gamma_P}{\widetilde \rho_P} \left(- 2 \overline \rho_P \widetilde \rho + \overline \rho_P^2 + \widetilde \rho_P^2\right) = - \frac{\overline \gamma_P}{\widetilde \rho_P} (\overline \rho_P - \widetilde \rho_P)^2 \le 0,   
\end{array}
\] 
where (a) holds in view of \eqref{eq:minpri}, (b) is due to the concavity of $\gamma_P$ and the convexity of $\rho_P$, while (c) follows from $\widetilde \gamma_P / \overline \gamma_P > \overline \rho_P / \widetilde \rho_P$, which is entailed by \eqref{eq:interstep}. Then, $\overline x$ is global optimal for problem \eqref{eq:mobour} and, thanks to Proposition \ref{th:ottimo-pareto}, it is a Pareto optimum for problem \eqref{eq:mobini}.     

\end{proof}
Leveraging Proposition \ref{th:statglob}, we are left to devise a procedure that allows one to compute efficiently a stationary point for problem \eqref{eq:mobour} with a corresponding positive value of the area $A$. Suffice it to rely on the projected gradient algorithm with a sufficiently small constant stepsize, having the forethought to choose a starting point $x^0 \in X$ such that $A(x^0) > 0$. We recall that the projected gradient algorithm is a sequential convex approximation \citep[for further details cf.][]{facchinei2014flexible,facchinei2015parallel,facchinei2016feasible,scutari2014parallel,scutari2017parallel1,scutari2017parallel2} method whose core step consists in the calculation of the projection of a gradient iteration on the feasible set. Interestingly, as for problem \eqref{eq:mobour}, thanks to the peculiar nature of the objective function $A$, the projection can be performed not on the overall feasible set $X$ but only on the much easier to handle simplex $\Delta$. We stress that the computation of the projection $P_\Delta(x + \tau \nabla A(x))$ does not require to solve an optimization problem: it can be easily performed by means of simple calculations (see Section \ref{sec:ptf2} for more details).

Preliminarily, we observe that, letting $L_{\nabla \gamma_P}$, $L_{\nabla \rho_P}$, $L_{\rho_P}$, $L_{\gamma_P}$ be the Lipschitz moduli for $\nabla \gamma_P$, $\nabla \rho_P$, $\gamma_P$ and $\rho_P$ on $X$, respectively, and $\rho_P^{\text{min}} \triangleq \min \{\rho_P(x) \, : \, x \in X\}$ and $\gamma_P^{\text{max}} \triangleq \max \{\gamma_P(x) \, : \, x \in X\}$,
\begin{equation}\label{eq:lip}
L \triangleq L_{\nabla \gamma_P} (\rho_P^{\text{ref}} - \rho_P^{\text{min}}) + L_{\nabla \rho_P} (\gamma_P^{\text{max}} - \gamma_P^{\text{ref}}) + 2 L_{\rho_P} L_{\gamma_P} 
\end{equation}
is easily seen to be a Lipschitz constant for $\nabla A$ on $X$.
With these considerations in mind, here we report the resulting scheme of the method specified for problem \eqref{eq:mobour}. 
  
\medskip

\begin{algorithm}[H]
\KwData{$\tau \in (0, \frac{2}{L})$, $x^{0} \in X$ with $A(x^0) > 0$,  $\nu \longleftarrow 0$\;}
\Repeat{
{\nlset{(S.1)} \If{$x^{\nu}$ {\emph{is stationary for}} \eqref{eq:mobour}}{{\bf{stop}} and {\bf{return}} $x^\nu$\;} \label{S.1}}
\nlset{(S.2)} set $x^{\nu+1} = P_\Delta \big(x^\nu + \tau \nabla A(x^\nu)\big)$, $\nu\longleftarrow\nu+1$\; \label{S.2}}{}
\caption{\label{algo}Modified projected gradient Algorithm}
\end{algorithm}

\noindent
To obtain the convergence of Algorithm \ref{algo}, we need to show that, despite we are projecting on $\Delta$, each iterate $x^\nu$ is such that $\gamma_P(x^\nu) > \gamma_P^{\text{ref}}$ and $\rho_P(x^\nu) < \rho_P^{\text{ref}}$. 
\begin{proposition}\label{th:converg}
Each iterate $x^\nu \in \Delta$ generated by Algorithm \ref{algo} is such that $x^\nu \in X$ and $A(x^\nu) > 0$ and each limit point of the sequence $(x^\nu)$ is global optimal for problem \eqref{eq:mobour} and, in turn, a global Pareto optimum for \eqref{eq:mobini}.
\end{proposition}
\begin{proof}
We prove the claim by induction. More precisely, we show that, if $x^\nu \in \Delta$ is such that $x^\nu \in X$ and $A(x^\nu) > 0$, then, for the next iterate $x^{\nu + 1} \in \Delta$, we have
\begin{equation}\label{eq:notworsen}
\rho_P(x^{\nu+1}) \le \rho_P^{\text{ref}} \enspace \text{or} \enspace \gamma_P(x^{\nu+1}) \ge \gamma_P^{\text{ref}}, 
\end{equation}
and, thus, since $A(x^{\nu + 1}) > A(x^\nu) > 0$ by standard reasonings \citep[cf. the proof of Proposition 2.3.2 in][]{Bert99}, $x^{\nu+1}$ must belong to $X$. Let us assume that \eqref{eq:notworsen} is not true: i.e.,
\begin{equation}\label{eq:notworsencontr}
\rho_P(x^{\nu+1}) > \rho_P^{\text{ref}} \enspace \text{and} \enspace \gamma_P(x^{\nu+1}) < \gamma_P^{\text{ref}}.
\end{equation}
By using the descent lemma \citep[Proposition A.24 in][]{Bert99} and \eqref{eq:notworsencontr},  recalling that $\rho_P(x^{\nu}) < \rho_P^{\text{ref}}$, $\gamma_P(x^{\nu}) > \gamma_P^{\text{ref}}$, we have
\begin{equation}\label{eq:intermrel}
\begin{array}{c}
0 < \rho_P(x^{\nu+1}) - \rho_P(x^{\nu}) \le \nabla \rho_P(x^\nu)\trt (x^{\nu + 1} - x^{\nu}) + \frac{L_{\nabla{\rho_P}}}{2} \|x^{\nu+1} - x^{\nu}\|^2,\\[5pt]
0 < \gamma_P(x^{\nu}) - \gamma_P(x^{\nu+1}) \le - \nabla \gamma_P(x^\nu)\trt (x^{\nu + 1} - x^{\nu}) + \frac{L_{\nabla \gamma_P}}{2} \|x^{\nu+1} - x^{\nu}\|^2.
\end{array}
\end{equation}
In turn, we get the following contradiction:
\[
\begin{array}{rcl}
\frac{1}{\tau} \|x^{\nu+1} - x^\nu\|^2 & \overset{(a)}{\le}& \nabla A(x^\nu)\trt (x^{\nu+1} - x^\nu)\\[5pt]
& = & \left[(\rho_P^{\text{ref}} - \rho_P(x^\nu)) \nabla \gamma_P(x^\nu) - (\gamma_P(x^\nu) - \gamma_P^{\text{ref}}) \nabla \rho_P(x^\nu) \right]\trt (x^{\nu+1} - x^{\nu})\\[5pt]  
& \overset{(b)}{<} & \left[(\rho_P^{\text{ref}} - \rho_P(x^\nu)) \frac{L_{\nabla \gamma_P}}{2} + (\gamma_P(x^\nu) - \gamma_P^{\text{ref}}) \frac{L_{\nabla \rho_P}}{2}\right] \|x^{\nu+1} - x^{\nu}\|^2\\[5pt]
& \overset{(c)}{<} & \frac{1}{\tau} \|x^{\nu+1} - x^\nu\|^2,
\end{array}
\]
where (a) is due to the classical sufficient ascent property (for $A$ at $x^\nu$) of the projected gradient direction ($x^{\nu+1} - x^\nu$) \citep[see, e.g., equation (2.53) in][]{Bert99}, (b) follows from \eqref{eq:intermrel} and (c) holds since
\[
\frac{1}{\tau} > \frac{L}{2} \ge (\rho_P^{\text{ref}} - \rho_P(x^\nu)) \frac{L_{\nabla \gamma_P}}{2} + (\gamma_P(x^\nu) - \gamma_P^{\text{ref}}) \frac{L_{\nabla \rho_P}}{2},
\]
which is true in view of $\tau \in (0, 2/L)$ and \eqref{eq:lip}. The convergence of the algorithm to a stationary point, with a corresponding positive area, follows standard arguments \citep[see, e.g., Section 2.3 in][]{Bert99}. The last claim is due to Propositions \ref{th:ottimo-pareto} and \ref{th:statglob}.
\end{proof}

\section{Numerical results}\label{sec:ptf2}
We study the behavior of our selection strategy using, as benchmark, the datasets presented in \cite{bruni2016real}: in particular, data consist in weekly returns time series for assets and indexes belonging to several major stock markets (see summary Table \ref{tab:dsets}).

As for the numerical testing of the approach, referring to \cite{markowitz1952portfolio}, we employ variance as risk measure. In our framework, having $n$ available assets, let $\gamma_P(x) = 100 \, m\trt x$ and $\rho_P(x) = 100 \, x\trt V x$ denote (in percentage) mean and variance of the return associated with portfolio $x \in \mathbb R^n$ in terms of mean $m \in \mathbb R^n$ and variance $\mathbb M_n \ni V \succeq 0$ of the return of the $n$ assets, respectively. We distinguish between portfolios by resorting to the optimization problem \eqref{eq:mobour} where we consider as reference point the nadir vector: hence, {\boldmath$\rho_P^{\text{ref}}$} is the risk (in percentage) of the portfolio with maximal gain and {\boldmath$\gamma_P^\text{ref}$} is the gain (in percentage) of the minimum variance portfolio. We also denote by {\boldmath$\rho_P^{\min}$} $\triangleq 100 \min \{x\trt V x \, : \, x \in X\}$ and by {\boldmath$\gamma_P^{\max}$} $\triangleq 100 \max_{i \in \{1,\ldots,n\}} \{m_i\}$, the ideal values (in percentage) of risk and gain, respectively. All these data are collected in Table \ref{tab:dsetsbis}.

The results of our tests are summarized in Table \ref{tab:res1}.
For each index we compare our approach ({\bf area}) with the classical $\varepsilon$-constraints scalarization method in which the gain {\boldmath$\gamma_P$} is required to be no worse than a fixed level.
In particular, we consider 5 different levels: namely, $\alpha (\gamma_P^{\max}-\gamma_P^{\text{ref}}) + \gamma_P^{\text{ref}}$ with $\alpha = 0.01$, $0.25$, $0.5$, $0.75$, and $0.99$. For each method we report: gain {\bf \boldmath$\gamma_P$} $\triangleq 100 \, m\trt x$, risk {\boldmath$\rho_P$} $\triangleq 100 \, x\trt V x$, area {\boldmath$A$} $\triangleq \big(${\boldmath{$\gamma_P - \gamma_P^{\text{ref}}$}}$\big) \, \big(${\boldmath{$\rho_P^{\text{ref}} - \rho_P$}}$\big)$, and with {\bf \#ptf\_a} the number of $x_i \geq 1e-3,$ with $x$ being the computed solution. Moreover,
$${\boldsymbol{\beta_1}} \triangleq \frac{\gamma_P^{\max} - \gamma_P(x)}{\gamma_P^{\max} - \gamma_P^{\text{ref}}}, \enspace {\boldsymbol{\beta_2}} \triangleq \frac{\rho_P(x) - \rho_P^{\min}}{\rho_P^{\text{ref}} - \rho_P^{\min}},$$
and, thus, $\boldsymbol{\|\beta\|}$, with $\beta \triangleq (\beta_1, \, \beta_2)\trt,$ denotes the normalized distance of the computed values for risk and gain from the ideal ones $(\rho_P^{\min}, \,\gamma_P^{\max}).$ 
The numerical tests confirm the theoretical insights of previous sections (see Figure \ref{fig:area} and Table \ref{tab:res1}). In particular, our approach actually provides a solution that maximizes the corresponding area. Furthermore, the observed behavior of $\|\boldsymbol{\beta}\|$ (see Table \ref{tab:res1}) clearly indicates that employing the area $A$ as performance measure, while enjoying the distinctive and nice features highlighted in Section \ref{sec:prel}, gives also values for risk and gain that turn out to be close to the ideal ones. Finally, in column {\bf improve}$|$worsen, we report the numerical evidences for the peculiar feature underlined in Proposition \ref{th:proper}. Specifically, we distinguish two cases: if the efficient portfolio $x$ (obtained by the $\varepsilon$-constraints scalarization) entails an improvement w.r.t. $x_A$ in terms of $\gamma_P$, then we indicate in the table the quantities
\[
\text{\bf improvement} \; {\boldsymbol{\gamma}}_{\mathbf{P}} = \left. \frac{{\boldsymbol{\gamma}}_{\mathbf{P}} {\mathbf{(x) -}} {\boldsymbol{\gamma}}_{\mathbf{P}}^{\mathbf{\text ref}}}{{\boldsymbol{\gamma}}_{\mathbf{P}} {\mathbf{(x_A) -}} {\boldsymbol{\gamma}}_{\mathbf{P}}^{\mathbf{\text ref}}} \; \right| \; \frac{\rho_P^{\text ref} - \rho_P(x_A)}{\rho_P^{\text ref} - \rho_P(x)} = {\text{worsening}} \; \rho_P, 
\] 
whereas, if the improvement is obtained in terms of  $\rho_P$, then we indicate in the table the quantities
\[
{\text{worsening}} \; \gamma_P = \left. \frac{\gamma_P(x_A) - \gamma_P^{\text ref}}{\gamma_P(x) - \gamma_P^{\text ref}} \; \right| \; \frac{{\boldsymbol{\rho}}_{\mathbf{P}}^{\mathbf{\text ref}} - {\boldsymbol{\rho}}_{\mathbf{P}} {\mathbf{(x)}}}{{\boldsymbol{\rho}}_{\mathbf{P}}^{\mathbf{\text ref}} - {\boldsymbol{\rho}}_{\mathbf{P}} {\mathbf{(x_A)}}} = \text{\bf improvement} \; {\boldsymbol{\rho}}_{\mathbf{P}}.
\] 
Confirming the theoretical insights of Proposition \ref{th:proper}, by moving from $x_A$ to $x$, the worsening in terms of an objective is always greater than the corresponding improvement in terms of the other objective. 

\begin{figure}[h!]
\centering
\includegraphics[scale=0.45]{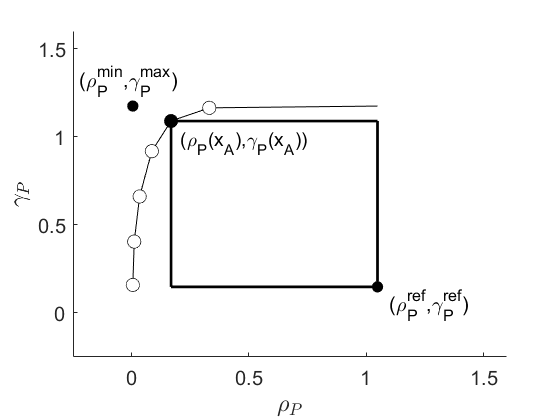}
\caption{Plot of results for the NASDAQComp data set}\label{fig:area}
\end{figure}

\begin{table}[h!]
\centering
\caption{Datasets synthetic description}\label{tab:dsets}
\vspace{5pt}
{\footnotesize{\begin{tabular}{r|ccc}
 \bf{Index} & \bf{\#assets} & \bf{\#weeks} & \bf{From-To}  \\\hline
 DowJones & 28 & 1363 & 2/1990 - 4/2016\\
 NASDAQ100 & 82 & 596 & 11/2004 - 4/2016\\
 FTSE100 & 83 & 717 & 7/2002 - 4/2016\\
 SP500 & 442 & 595 & 11/2004 - 4/2016\\
 NASDAQComp & 1203 & 685 & 2/2003 - 4/2016\\
 FF49Industries & 49 & 2325 & 7/1969 - 7/2015\\
   \hline
\end{tabular}}}
\end{table}

\noindent All experiments have been carried out on an Intel Core i7-4702MQ CPU @ 2.20GHz x 8 with Ubuntu 14.04 LTS 64-bit and by using Matlab R2017b. When addressing the nonconvex problem \eqref{eq:mobour}, we resort to the modified projected gradient Algorithm \ref{algo}. In particular, the procedure starts from $x^0 \in \Delta$ such that $\gamma_P(x^\nu) > \gamma_P^{\text{ref}}$ and $\rho_P(x^\nu) < \rho_P^{\text{ref}}$ and, thus, $A(x^0) > 0$ (see Table \ref{tab:dsetstris}): $x^0$ is computed as the portfolio with the whole budget invested on the asset with maximal ratio $\gamma_P / \rho_P$. We use a stepsize $\tau = 0.1$ and, as stopping criterion (cf. step \ref{S.1} in Algorithm \ref{algo}), the value for the stationarity measure $\|x^{\nu+1} - x^\nu\|_\infty$ is required to be smaller than $1e-5$. Furthermore, we rely on the Matlab function \verb projsplx  in order to compute the projection on the simplex $\Delta$ (cf. step \ref{S.2} in Algorithm \ref{algo}). The quadratic problems, obtained by means of the $\varepsilon$-constraints method, are solved by using \verb quadprog  with \verb interior-point-convex  algorithm. In Table \ref{tab:dsetstris} we also report the total number of iterations {\bf{\#iter}} and the elapsed time {\bf{elaps\_time}} (in seconds) needed to satisfy the stopping criterion in step \ref{S.2}. We stress that, for each stock market considered, the amount of time to solve problem \eqref{eq:mobour} is below half a second.   

\renewcommand{\abstractname}{Acknowledgements}
\begin{abstract}
The authors are grateful to F. Tardella for his insightful suggestions and comments on the contents of this work. 
\end{abstract}

\begin{table}[h!]
\centering
\caption{Datasets ideal and nadir values}\label{tab:dsetsbis}
\vspace{5pt}
{\footnotesize{\begin{tabular}{r|cccc}
 \bf{Index} & \boldmath${\gamma_P^\text{\bf ref}}$ & \boldmath${\gamma_P^{\max}}$ & \boldmath${\rho_P^{\min}}$ & \boldmath${\rho_P^\text{\bf ref}}$ \\\hline
 DowJones & 0.214 & 0.605 & 0.040 & 0.347 \\
 NASDAQ100 & 0.242 & 1.030 & 0.039 & 0.676 \\
 FTSE100 & 0.254 & 0.802 & 0.030 & 0.649  \\
 SP500 & 0.190 & 1.032 & 0.022 & 0.677  \\
 NASDAQComp & 0.148 & 1.174 & 0.005 & 1.048  \\
 FF49Industries & 0.325 & 0.544 & 0.029 & 0.095  \\
   \hline
\end{tabular}}}
\end{table}

\begin{table}[h!]
\centering
\caption{Numerical results for the mean-variance model \label{tab:res1}}
\vspace{5pt}
{\footnotesize{\begin{tabular}{lr|cccccccc}
 & \bf{Method} & \boldmath$\gamma_P$ & \boldmath$\rho_P$ & \boldmath$A$ & $\boldsymbol{\beta}_1$ &  $\boldsymbol{\beta}_2$ & $\|\boldsymbol{\beta}\|$ & {\bf{improve}}$|$worsen & \bf{\#ptf\_a} \\  \hline
 \multirow{6}{*}{\begin{sideways}\bf{\tiny{DowJones}}\end{sideways}} & {\bf area} &  0.542 & 0.129 & 0.071 & 0.162 & 0.290 & 0.332 & 1.000 $|$ 1.000 & 6  \\
 & \boldmath$\gamma_P$ $\geq$ 0.218 & 0.218 & 0.040 & 0.001 & 0.990 & 0.000 & 0.990 & 82.00 $|$ {\bf{1.408}} & 15  \\
 & \boldmath$\gamma_P$ $\geq$ 0.312 & 0.312 & 0.048 & 0.029 & 0.750 & 0.016 & 0.750 & 3.347 $|$ {\bf{1.371}} & 14   \\
 & \boldmath$\gamma_P$ $\geq$ 0.410 & 0.410 & 0.078 & 0.054 & 0.500 & 0.056 & 0.503 & 1.673 $|$ {\bf{1.234}}& 11  \\
 & \boldmath$\gamma_P$ $\geq$ 0.508 & 0.508 & 0.110 & 0.070 & 0.250 & 0.141 & 0.287 & 1.116 $|$ {\bf{1.087}} & 8  \\
 & \boldmath$\gamma_P$ $\geq$ 0.602 & 0.602 & 0.254 & 0.036 & 0.010 & 0.847 & 0.847 & {\bf{1.183}} $|$ 2.344 & 2  \\
 \hline
 \multirow{6}{*}{\begin{sideways}\bf{\tiny{NASDAQ100}}\end{sideways}} & {\bf area} &  0.918 & 0.174 & 0.339 & 0.142 & 0.212 & 0.255 & 1.000 $|$ 1.000 & 7  \\
 & \boldmath$\gamma_P$ $\geq$ 0.250 & 0.250 & 0.039 & 0.005 & 0.774 & 0.000 & 0.990 & 84.50 $|$ {\bf{1.269}} & 11  \\
 & \boldmath$\gamma_P$ $\geq$ 0.439 & 0.439 & 0.049 & 0.123 &  0.750 & 0.000 & 0.750 & 3.431 $|$ {\bf{1.249}}  & 19  \\
 & \boldmath$\gamma_P$ $\geq$ 0.636 & 0.636 & 0.075 & 0.237 & 0.500 & 0.028 & 0.501 & 1.716 $|$ {\bf{1.197}} & 13  \\
 & \boldmath$\gamma_P$ $\geq$ 0.833 & 0.833 & 0.129 & 0.323 & 0.250 & 0.105 & 0.272 & 1.144 $|$ {\bf{1.090}} & 10  \\
 & \boldmath$\gamma_P$ $\geq$ 1.022 & 1.022 & 0.578 & 0.076 & 0.010 & 0.770 & 0.770 & {\bf{1.154}} $|$ 5.122 & 2  \\
 \hline
 \multirow{6}{*}{\begin{sideways}\bf{\tiny{FTSE100}}\end{sideways}} & {\bf area} &  0.680 & 0.157 & 0.210 & 0.222 & 0.205 & 0.302 & 1.000 $|$ 1.000 & 4  \\
 & \boldmath$\gamma_P$ $\geq$ 0.259 & 0.259 & 0.030 & 0.003 & 0.990 & 0.000 & 0.990 & 85.20 $|$ {\bf{1.258}} &  23  \\
 & \boldmath$\gamma_P$ $\geq$ 0.391 & 0.391 & 0.037 & 0.084 & 0.750 & 0.012 & 0.750 & 3.109 $|$ {\bf{1.244}} & 15  \\
 & \boldmath$\gamma_P$ $\geq$ 0.528 & 0.528 & 0.066 & 0.160 & 0.500 & 0.058 & 0.503 & 1.555 $|$ {\bf{1.185}} & 9  \\
 & \boldmath$\gamma_P$ $\geq$ 0.665 & 0.665 & 0.141 & 0.209 & 0.250 & 0.179 & 0.307 & 1.036 $|$ {\bf{1.032}} & 4  \\
 & \boldmath$\gamma_P$ $\geq$ 0.796 & 0.796 & 0.596 & 0.029 & 0.010 & 0.914 & 0.914 & {\bf{1.272}} $|$ 9.283  & 2  \\
 \hline
 \multirow{6}{*}{\begin{sideways}\bf{\tiny{SP500}}\end{sideways}} & {\bf area} &  0.914 & 0.169 & 0.368 & 0.141 & 0.225 & 0.265 & 1.000 $|$ 1.000 & 7  \\
 & \boldmath$\gamma_P$ $\geq$ 0.198 & 0.198 & 0.022 & 0.006 & 0.990 & 0.000 & 0.990 & 90.50 $|$ {\bf{1.289}} & 25  \\
 & \boldmath$\gamma_P$ $\geq$ 0.400 & 0.400 & 0.030 & 0.136 & 0.750 & 0.013 & 0.750 & 3.448 $|$ {\bf{1.274}} & 20  \\
 & \boldmath$\gamma_P$ $\geq$ 0.611 & 0.611 & 0.059 & 0.260 & 0.500 & 0.057 & 0.503 & 1.720 $|$ {\bf{1.216}} & 16  \\
 & \boldmath$\gamma_P$ $\geq$ 0.822 & 0.822 & 0.122 & 0.351 & 0.250 & 0.153 & 0.293 & 1.145 $|$ {\bf{1.092}} & 11  \\
 & \boldmath$\gamma_P$ $\geq$ 1.024 & 1.024 & 0.570 & 0.089 & 0.010 & 0.837 & 0.837 & {\bf{1.152}} $|$ 4.748  & 2  \\
 \hline
 \multirow{6}{*}{\begin{sideways}\bf{\tiny{NASDAQComp}}\end{sideways}} & {\bf area} &  1.089 & 0.168 & 0.828 & 0.083 & 0.156 & 0.177 & 1.000 $|$ 1.000 &  14  \\
 & \boldmath$\gamma_P$ $\geq$ 0.159 & 0.159 & 0.005 & 0.011 & 0.990 & 0.000 & 0.990 & 85.54 $|$ {\bf{1.185}} & 79   \\
 & \boldmath$\gamma_P$ $\geq$ 0.405 & 0.405 & 0.011 & 0.266 & 0.750 & 0.006 & 0.750 & 3.661 $|$ {\bf{1.178}} & 81   \\
 & \boldmath$\gamma_P$ $\geq$ 0.661 & 0.661 & 0.034 & 0.521 & 0.500 & 0.028 & 0.501 & 1.834 $|$ {\bf{1.152}} & 67   \\
 & \boldmath$\gamma_P$ $\geq$ 0.918 & 0.918 & 0.086 & 0.740 & 0.250 & 0.078 & 0.262 & 1.222 $|$ {\bf{1.093}} & 38   \\
 & \boldmath$\gamma_P$ $\geq$ 1.164 & 1.164 & 0.331 & 0.729 & 0.010 & 0.312 & 0.312 & {\bf{1.080}} $|$ 1.227  & 3   \\
 \hline
 \multirow{6}{*}{\begin{sideways}\bf{\tiny{FF49Industries}}\end{sideways}} & {\bf area} &  0.465 & 0.051 & 0.006 & 0.361 & 0.332 & 0.490 & 1.000 $|$ 1.000 & 8  \\
 & \boldmath$\gamma_P$ $\geq$ 0.327 & 0.327 & 0.029 & 0.000 & 0.990 & 0.000 & 0.990 & 70.00 $|$ {\bf{1.500}} & 6   \\
 & \boldmath$\gamma_P$ $\geq$ 0.380 & 0.380 & 0.033 & 0.003 & 0.750 & 0.058 & 0.752 & 2.545 $|$ {\bf{1.409}} & 8   \\
 & \boldmath$\gamma_P$ $\geq$ 0.435 & 0.435 & 0.042 & 0.006 & 0.500 & 0.195 & 0.537 & 1.273 $|$ {\bf{1.204}} & 8   \\
 & \boldmath$\gamma_P$ $\geq$ 0.489 & 0.489 & 0.059 & 0.006 & 0.250 & 0.454 & 0.518 & {\bf{1.171}} $|$ 1.222  & 6   \\
 & \boldmath$\gamma_P$ $\geq$ 0.542 & 0.542 & 0.090 & 0.001 & 0.010 & 0.926 & 0.926 & {\bf{1.550}} $|$ 8.800  &  2   \\
 \hline
 \end{tabular}}}
\end{table}

\begin{table}[h!]
\centering
\caption{Algorithmic details}\label{tab:dsetstris}
\vspace{5pt}
{\footnotesize{\begin{tabular}{r|cccccc}
 \bf{Index} & \boldmath${\gamma_P(x^0)}$ & \boldmath${\rho_P(x^0)}$ & \boldmath${A(x^0)}$ & \bf{\#iter} & \bf{elaps\_time} \\\hline
 DowJones & 0.246 & 0.089 & 0.008 & 659 & 0.030 \\
 NASDAQ100 & 0.325 & 0.090 & 0.048 & 206 & 0.013 \\
 FTSE100 & 0.293 & 0.065 & 0.023 & 379 & 0.020 \\
 SP500 & 0.299 & 0.055 & 0.068 & 183 & 0.129 \\
 NASDAQComp & 0.382 & 0.096 & 0.223 & 89 & 0.433 \\
 FF49Industries & 0.378 & 0.039 & 0.003 & 10951 & 0.341  \\
   \hline
\end{tabular}}}
\end{table}


\bibliographystyle{plainnat}
\bibliography{Surbib}

\end{document}